\documentclass[12pt]{article}
\usepackage{amsmath,amsfonts,amsthm,filecontents}

\usepackage[a4paper]{geometry}
\newtheorem{theorem}{Theorem}
\newtheorem{lemma}[theorem]{Lemma}

\newcommand{\N}{\mathbb N}

\newcommand{\BB}{\mathcal{B}}

\DeclareMathOperator{\vol}{vol}
\DeclareMathOperator{\exc}{exc}

\begin{document}
\sloppypar

\title{A Lower Bound for the Discrepancy of a Random Point Set}
\author{Benjamin Doerr\\
Max Planck Institute for Computer Science,\\
Campus E1 4\\
66123 Saarbr\"ucken, Germany}

\newcommand\eps{\varepsilon}
\maketitle

\begin{abstract}
  We show that there are constants $k, K > 0$ such that for all $N, s \in \N$, $s \le N$, the point set consisting of $N$ points chosen uniformly at random in the $s$-dimensional unit cube $[0,1]^s$ with probability at least $1 - e^{-ks}$ admits an axis-parallel rectangle $[0,x] \subseteq [0,1]^s$ containing $K \sqrt{sN}$ points more than expected. Consequently, the expected star discrepancy of a random point set is of order $\sqrt{s/N}$. 
\end{abstract}

\section{Introduction}

Discrepancy theory~\cite{BeckS95} deals with different types of uniformity questions including balanced colorings of hypergraphs, rounding problems, or balancing games. A discrepancy topic of particular interest in numerical analysis is the study of uniformly distributed point sets and sequences~\cite{Niederreiter92,DrmotaT97,DickP10}. Such constructions are the basis of quasi-Monte Carlo integration, the degree of their uniformity can be used to derive upper bounds for the integration error. 

Among several types of uniformity notions, the one of the star discrepancy seems to be the most common one. Let $N, s \in \N$. Let $P \subseteq [0,1]^s$ with $|P| = N$. For $x \in [0,1]^s$, let us call the set $[0,x] := \prod_{i = 1}^s [0,x_i]$ a \emph{box} (other names in use are anchored boxes or corners), and denote by $\BB := \{[0,x] \mid x \in [0,1]^s\}$ the set of all these boxes. Denoting the Lebesgue measure of a measurable set $B$ simply by $\vol(B)$,  the \emph{star discrepancy} of $P$ now is defined by \[D^*(P) := \sup_{B \in \BB} \big|\tfrac 1N |P \cap B| - \vol(B)\big|.\] It is thus a measure of how well $P$ satisfies the aim of being uniformly distributed with respect to all boxes---in the sense that each box contains a fraction of points equal to its share of the volume of the whole unit cube. 

The connection to numerical integration is made, among others, by the Koksma-Hlawka inequality~\cite{Koksma42,Hlawka61}, which states that the integral $\int_{[0,1]^s} f(x) dx$ is well approximated by the average $\frac 1N \sum_{p \in P} f(p)$ when $P$ has low discrepancy:  the approximation error can be bounded by the product of the star discrepancy of $P$ and a certain variation measure of $f$. This and similar results explain the enormous attention low-discrepancy point sets and sequences attracted.

The classic view on low-discrepancy point sets is to regard the asymptotics in $N$, the number of points, assuming the dimension $s$ to be fixed. Interestingly, random point sets are far from having the optimal behavior---their expected discrepancy is easily seen to be at least of the order of $N^{-1/2}$. A number of deep results of the last good 70 years provide point sets having a discrepancy of order $(1/N) (\log N)^{s-1}$, see~\cite{Niederreiter92,DrmotaT97,DickP10,Matousek10}. 

More recently, it was noted that this discrepancy behavior, and in particular, taking the dimension $s$ as a constant, is not very useful in many practical applications. If $s$ equals 360, as in some applications in finance, the number $N$ of points must be prohibitively large to let the term $(1/N) (\log N)^{s-1}$ sink below the trivial bound of $1$. Consequently, Heinrich, Novak,  Wasilkowski, and Wo\'{z}niakowski~\cite{HeinrichNWW01} started the quest for bounds and construction that have a better, in particular polynomial, dependence on $s$. Among other results, they show that the minimal star discrepancy of an $N$-point set in the $s$-dimensional unit cube is $O(\sqrt{s/N})$. Surprisingly, this bound is obtained by a random point set already. More precisely, the proof of  Theorem~3 in~\cite{HeinrichNWW01} shows that a random point set with probability $1 - \exp(-\Theta(c))$ has a discrepancy of at most $c \sqrt{s/N}$. This immediately implies that the expected discrepancy of a random point set is $O(\sqrt{s/N})$. 
See~\cite{AistleitnerH} for an elementary proof of these results that also gives good values for the implicit constants. See~\cite{Gnewuch12} for a recent survey on discrepancy results with explicit dependence on the dimension $s$.

No matching lower bounds for the minimal star discrepancy are known, the best one is $\Omega(s/N)$ by Hinrichs~\cite{Hinrichs04} (of course assuming $s = O(N)$). Closing this gap is one of the big open problems in this field.

Surprisingly, not even a lower bound for the discrepancy of a random point set is known. This is the objective of this note, where we give a simple proof that the upper bounds given in~\cite{HeinrichNWW01,AistleitnerH} are asymptotically tight.

\begin{theorem}\label{thm}
  There is an absolute constant $K$ such that the following is true. Let $N, s \in \N$ such that $s \le N$. Let $P$ be a set of $N$ points chosen independently and uniformly at random from $[0,1]^s$. Then the expected star discrepancy satisfies $E[D^*(P)] \ge K \sqrt{s/N}$. The probability that  $D^*(P)$ is less than $K \sqrt{s/N}$, is at most $\exp(-\Theta(s))$.
\end{theorem}

To keep this note simple, we did not try to find good explicit values for the constants. 

Calculating the precise value of $E[D^*(P)]$ for fixed $s$ and $N$ seems to be a very difficult problem. To illustrate this fact, we just note that by Donsker's theorem we have for any fixed $s$ that
\begin{equation} \label{browniansh}
E[D^*(P)] \sqrt{N} \to E \Big[ \sup_{(t_1, \dots, t_s) \in [0,1]^s} \left| B (t_1, \dots, t_s) \right| \Big]  \qquad \textrm{as $N \to \infty$},
\end{equation}
where $B$ denotes the $s$-dimensional (standard) pinned Brownian sheet, and the precise value of the expression on the right-hand side of \eqref{browniansh} is unknown and considered to be a difficult problem (for $s>1$). See~\cite{vdVaartW96} for more details.

Quantities similar to $E[D^*(P)]$, in the more general form $E [ \sup_{f \in \mathcal{F}} | \sum_{p \in P} f(p) | ]$ for a class of functions $\mathcal{F}$, play an important role in nonparametric statistics, where they appear in large deviations inequalities for empirical processes (see for example \cite{tala,mass}). Consequently for applications it is desirable to find good estimates for such quantities. 

Finally, let us note that there are results for the limit behavior (for $N$ tending to infinity) of the expected $L_p$-star discrepancy of random point sets, see~\cite{HinrichsW12} and the references therein. We are not aware, though, of any results in these areas which would imply our discrepancy result.

\section{Proof}

The proof of Theorem~\ref{thm} is elementary and only relies on a well-known fact, namely that binomially distributed random variables with constant probability deviate from their expectation (in both directions) by an amount of order the square root of the expectation. 

We use this fact as follows. Starting with the box $B = [0,x]^s$ being the full box, that is, $x = (1, \ldots, 1)$, for $i$ from $1$ to $s$ sequentially we reduce $x_i$ from $1$ to $1 - 1/s$ if this increases the excess of points in $B$. By the above fact, in each such iteration with positive probability we gain an excess of $\Theta(\sqrt{N/s})$ points in $B$, finally leading to a box having $\Theta(\sqrt{sN})$ points more than it should.

More precisely, we use the following version of the above-mentioned probabilistic fact. To keep this note self-contained, we give an elementary proof of it in the appendix. Again, none of the constants has been optimized. 

\begin{lemma}\label{lprob}
  Let $X$ be a random variable distributed according to a binomial distribution with parameters $n \ge 16$ and $1/n \le p \le 1/4$. Then  $\Pr[X \le pn - \sqrt{pn}/2] \ge 3/160$.
\end{lemma}

To prove the main result, let us for a given $N$-point set $P$ denote by \[\exc(B) := |P \cap B| - N \vol(B)\] the excess of points in a measurable set $B$. Hence this is a signed discrepancy notion without normalization by $N$. 

Let $P$ be a set of $N$ random points chosen independently and uniformly in $[0,1]^s$. Since we do not care about the constant in our main result, we may assume that $N \ge 64$. For the same reason, for $s < 4$ we may simple regard any box $B$ of volume $3/4$ and invoke Lemma~\ref{lprob} to see that with probability at least $3/160$, its complement contains $\sqrt{N/4}/2$ points less than expected, implying $\exc(B) \ge \sqrt{N/4}/2$. Finally, we may assume that $s \le N/4$. If $s$ is larger (but still at most $N$ as assumed), we may project $P$ onto its first $s' = \lfloor N/4 \rfloor$ coordinates, apply the result to find a box $B' \subseteq [0,1]^{s'}$ with large excess, and note that $B := B' \times [0,1]^{s - s'}$ is a box in $[0,1]^s$ having the same excess. Hence we may assume in the following that $N \ge 64$ and $4 \le s \le N/4$.

We will now inductively define numbers $x_1, \ldots, x_s \in \{1 - 1/s, 1\}$ such that for $i = 0, \ldots, s$, the box $B_i := \prod_{j = 1}^i [0,x_j] \times [0,1]^{s-i}$ has an expected excess of order $i \sqrt{N/s}$ and surely has a nonnegative excess. We observe that any choice of  $x_1, \ldots, x_s$ gives $\vol(B_i) \ge (1  - 1/s)^s \ge 1/4$.

Note that $B_0 = [0,1]^s$ is already defined and has an excess of zero. Assume that for some $0 \le i < s$ we have fixed $x_1, \ldots, x_i$, and consequently, $B_0, \ldots, B_i$, and that $B_i$ has a nonnegative excess. Given all this, $B_i$ contains $N_i := |B_i \cap P| \ge N \vol(B_i) \ge N/4$ points, all whose $(i+1)$-st to $s$-th coordinate are independently and uniformly distributed in $[0,1]$. 

Consequently, the rectangle $C_{i+1} := \prod_{j = 1}^i [0,x_j] \times (1-1/s,1] \times [0,1]^{s-i-1} \subseteq B_i$ contains each of these $N_i$ points with probability $1/s$. By Lemma~\ref{lprob}, with probability at least $3/160$, the rectangle $C_{i+1}$ contains at least $\sqrt{N_i/s}/2$ points less than the expected value $N_i/s$. In this case, put $x_{i+1} := 1 - 1/s$, implying $B_{i+1} = B_i \setminus C_{i+1}$, otherwise put $x_{i+1} = 1$, implying $B_{i+1} = B_i$. In the first case, the excess of $B_{i+1}$ satisfies
\begin{align}
	\exc(B_{i+1}) 
	&= |P \cap B_i| - |P \cap C_{i+1}| -  N (1 - 1/s)\vol(B_i) \nonumber\\
	&\ge N_i - (N_i/s - \sqrt{N_i/s}/2) - N (1 - 1/s)\vol(B_i) \nonumber\\
	&=  (1-1/s) \exc(B_i) + \sqrt{N_i/s}/2 \nonumber\\
	&\ge (1-1/s) \exc(B_i) + \sqrt{N/(16s)}.\label{eq1}
\end{align}
In the second case, $\exc(B_i)$ and $\exc(B_{i+1})$ are trivially equal. In both cases, $B_{i+1}$ has a nonnegative excess.

Assume that all $x_i$ and $B_i$ were constructed in this fashion. Let $k$ be the number of $x_i$ which are equal to $1-1/s$. Then by repeated application of (\ref{eq1}), the excess of $B_s$ is at least $k (1 - 1/s)^k \sqrt{N/(16s)}  \ge k \sqrt{N/s}/16$. Since each $x_i$ independently with probability at least $3/160$ is $1 - 1/s$, we have $E(k) \ge 3s/160$. Consequently, the expected excess of $B_s$ is at least $(3/2560) \sqrt{sN}$. Also, by a simple Chernoff bound argument (e.g.,~\cite[Corollary~A.1.14]{AlonS08}), the probability that $k$ is less than, say, $\frac 12 3s / 160$, a necessary condition for the excess being less than $\frac 12 (3/2560) \sqrt{sN}$, is at most $\exp(-\Theta(s))$. 

This concludes the proof.

\subsection*{Acknowledgments}

The author is thankful to Christoph Aistleitner, Michael Gnewuch, Aicke Hinrichs, and Erich Novak for several useful comments. Special thanks to Christoph Aistleitner for contributing the penultimate and antepenultimate paragraphs of the introduction.

\bibliographystyle{alpha}

\appendix

\section{Proof of the Lemma~\ref{lprob}}

To keep this note self-contained, we quickly prove the probabilistic estimate of Lemma~\ref{lprob}. We start with a slightly stronger statement for the case of fair coin flips.

\begin{lemma}\label{lcoin}
  Let $X$ be a binomial random variable with parameters $n \ge 0$ and $p = 1/2$. Then $\Pr[X \le n/2 - (1/2)\sqrt{n/2}] \ge 1/8$.
\end{lemma}

\begin{proof}
  For $n \le 3$, we have $\Pr[X \le n/2 - (1/2)\sqrt{n/2}] \ge \Pr[X = 0] = 2^{-n} \ge 1/8$, for $n = 4$, we estimate $\Pr[X \le n/2 - (1/2)\sqrt{n/2}] = \Pr[X \in \{0,1\}] = 5/16 \ge 1/8$.

By the fact that $\Pr[X = i] = \Pr[X = n-i]$ is symmetric, we have $\Pr[X \le  n/2 - (1/2) \sqrt{n/2}] = (1/2) (1 - \Pr[|X - n/2| < (1/2)\sqrt{n/2}])$. Consequently, it suffices to show that $P_n := \Pr[|X - n/2| < (1/2)\sqrt{n/2}] \le 3/4$. Let $A_n := |\{i \in \N \mid |i - n/2| < (1/2)\sqrt{n/2}\}|$. Since the central binomial coefficient is the largest, we have $P_n \le A_n 2^{-n} \binom{n}{\lfloor n/2 \rfloor}$. We use the well-known estimate $\binom{n}{\lfloor n/2 \rfloor} \le 2^n / \sqrt{\pi n / 2}$ valid for all $n \ge 1$. 

For $5 \le n \le 9$, we have $A_n \le 2$ and thus $P_n \le 2 / \sqrt{\pi n / 2} \le 3/4$. For $10 \le n \le 18$, we have $A_n \le 3$ and thus $P_n \le 3 \cdot 2^{-n} \binom{n}{\lfloor n/2 \rfloor}$. This is easily seen to be less than $3/4$ by taking the precise value $\binom{n}{\lfloor n/2 \rfloor} = 252$ for $n = 10$ and the estimate $\binom{n}{\lfloor n/2 \rfloor} \le 2^n / \sqrt{\pi n / 2}$ for $11 \le n \le 18$. For $n \ge 19$, we have $A_n \le \sqrt{n/2} + 1$ and consequently $P_n \le (\sqrt{n/2}+1)/\sqrt{\pi n /2} \le (1 + 1/\sqrt{n/2}) / \sqrt{\pi} \le 3/4$.
\end{proof}

We use the result above to give an elementary proof of Lemma~\ref{lprob} via a detour through the binomial distribution with parameters $n$ and $2p$.

\begin{proof}
Let $Y$ be a random subset of $[n] := \{1, \ldots, n\}$ such that for all $i \in [n]$ independently, we have $\Pr[i \in Y] = 2p$. Let $Z$ be a random subset of $Y$ such that for all $i \in Y$ independently, we have $\Pr[i \in Z] = \frac 12$. Clearly, $z := |Z|$ follows a binomial distribution with parameters $n$ and $p$. 
  
  Since $y := |Y|$ is binomially distributed with parameters $n$ and $2p$,  its median is at most $\lceil 2pn \rceil$. Consequently, the probability that $y$ is at most $2pn$, is at least $1/2 - \Pr[y = \lceil 2pn \rceil]$. The binomial distribution has the property that $B_{n,p}(k)$ is maximal for $p = k/n$, and that $B_{n,k/n}(k)$ is decreasing for $k = 0, \ldots, \lfloor n/2\rfloor$ and increasing for $k = \lceil n/2 \rceil, \ldots, n$. Consequently, since $1/n \le p \le 1/4$ and $n \ge 16$, we have $\Pr[y = \lceil 2pn \rceil] \le B_{n,\lceil 2pn \rceil/n}(\lceil 2pn \rceil) \le B_{n,2/n}(2) \le 2 (1 - 2/n)^{n-2} \le 2 \exp(-2 (n-2)/n) \le 0.35$. Hence $\Pr(y \le 2pn) \ge 0.15$.
  
Conditional on the outcome of $y$, and assuming $1 \le y \le 2pn$, by Lemma~\ref{lcoin} we have 
\begin{align*}
	\Pr[z \le pn - (1/2) \sqrt{pn}] \ge \Pr[z \le y/2 - (1/2) \sqrt{y/2}] \ge 1/8.
\end{align*}
Here we used that $x \mapsto x/2 - (1/2)\sqrt{x/2}$ is increasing for $x \ge 1/8$. Note that if $y = 0$, we trivially have $\Pr[z \le pn - \sqrt{pn}] = 1 \ge 1/8$. Consequently,
\begin{align*}
	\Pr[z \le pn - (1/2) \sqrt{pn}] &\ge \sum_{i = 0}^{2pn} \Pr[y = i] \Pr[z \le pn - (1/2)\sqrt{pn} \mid y = i] \\
	& \ge (1/8) \Pr[y \le 2pn] \ge (1/8) \cdot 0.15 =  3/160.
\end{align*}
\end{proof}

\end{document}